\documentclass[a4paper,reqno,12pt]{article}
\usepackage{t1enc} 
\usepackage[margin=1in]{geometry}
\usepackage{xcolor}
\usepackage{graphicx}
\usepackage{amsmath,amsthm,amssymb}

\newcommand{\ed}{\stackrel{(d)}{=}}
\def\p{\mathbb{P}}
\def\e{\mathbb{E}}
\def\ind{\mbox{\rm 1\hspace{-0.04in}I}}

\theoremstyle{plain}
\newtheorem{theorem}{Theorem}

\newtheorem{corollary}{Corollary}

\theoremstyle{definition}

\numberwithin{equation}{section}

\title{Short proofs in extrema of spectrally\\ one sided L\'e{}vy processes}

\author{Lo\"i{}c Chaumont\thanks{LAREMA UMR CNRS 6093, Universit\'e d'Angers, 2, Bd Lavoisier 
Angers Cedex 01, 49045, France. Email: loic.chaumont@univ-angers.fr}  \and 
Jacek Ma{\l}ecki\thanks{Faculty of Pure and Applied Mathematics, Wroc{\l}aw University of Science and Technology, 
Wybrze\.ze Wyspia\'nskiego 27, 50-370 Wroc{\l}aw, Poland. Email: jacek.malecki@pwr.edu.pl
\newline
\indent J. Ma{\l}ecki is supported by the Polish National Science Centre (NCN) grant no. 2015/19/B/ST1/01457.}}

\date{\today}

\begin{document}

\maketitle

\begin{abstract} 
We provide short and simple proofs of the continuous time ballot theorem for processes with cyclically interchangeable increments 
and Kendall's identity for spectrally positive L\'evy processes. We obtain the later result as a direct consequence of 
the former. The ballot theorem is extended to processes having possible negative jumps. Then we prove through straightforward 
arguments based on the law of bridges and Kendall's identity, Theorem 2.4 in \cite{mpp} which gives an expression for the law 
of the supremum of spectrally positive L\'evy processes. An analogous formula is obtained for the supremum of spectrally negative 
L\'evy processes.
\end{abstract}



%
%

\section{Introduction}\label{sec:intro}

The series of notes from J. Bertrand \cite{be}, \'E. Barbier \cite{ba} and D. Andr\'e \cite{an} which appeared in 1887 has inspired 
an extensive literature on the famous ballot theorem for discrete and continuous time processes. In the same year, the initial 
question raised by J. Bertrand was related by himself  to the ruin problem. Using modern formalism, it can be stated in terms of the 
simple random walk $(S_n)_{n\ge0}$  as follows:
\begin{equation}\label{7266}
\p(T_k=n\,|\,S_n=-k)=\frac kn,\;\;\;k,n\ge1,
\end{equation}
where $T_k=\inf\{j:S_j=-k\}$. The first substantial extension was obtained in 1962 by L. Tak\'acs \cite{bib:ta62}  who proved that 
identity (\ref{7266}) is actually satisfied if $(S_n)_{n\ge0}$ is any downward skip free sequence with interchangeable increments 
such that $S_0=0$. Then the same author considered this question in continuous time and proved the identity 
\begin{equation}\label{2408}
\p(T_x=t\,|\,X_t=-x)=\frac{x}t\,,\;\;\;x,t>0,
\end{equation}
where $(X_s,0\le s\le t)=(Y_s-s,0\le s\le t)$ and $(Y_s,0\le s\le t)$ is an increasing continuous time stochastic process with cyclically interchangeable increments  and $T_x=\inf\{s:X_s=-x\}$, see  \cite{bib:ta67}. The first step of this note is to provide a short and 
elementary proof of a more general result than identity (\ref{2408}) which also applies to processes with negative jumps.\\

Identity (\ref{2408}) cannot be extended to all continuous time processes with cyclically interchangeable increments. 
A problem appears when the process has unbounded variation. In particular, if $(X_s,0\le s\le t)$ is a spectrally positive L\'evy 
process with unbounded variation, then we can check that $\p(T_x=t\,|\,X_t=-x)=0$. However, by considering the process
on the whole half line, it is still possible to compare the measures $\p(T_{x}\in {\rm d}t)\,{\rm d}x$ and 
$\p(-X_t\in {\rm d}x)\,{\rm d}t$ on $(0,\infty)^2$ in order to obtain the following analogous result:
\begin{equation}\label{8374}
\p(T_{x}\in dt)\,{\rm d}x=\frac xt\p(-X_t\in {\rm d}x)\,{\rm d}t\,.
\end{equation}
Identity (\ref{8374}) was first obtained in the particular case of compound Poisson processes by D.~Kendall in \cite{bib:ken57}
where the problem of the first emptiness of a reservoir is solved. It has later been extended by J.~Keilson in \cite{bib:ke63} and 
A.A.~Borovkov in \cite{bib:bo65} to all spectrally positive L\'evy processes. Since then several proofs have been given using 
fluctuation identities, chap.~VII of J.~Bertoin's book \cite{bib:jb96} or martingale identities and change of measures, K. Borovkov 
and Z. Burq  \cite{bib:bb01}.  We shall see in the next section that identity (\ref{8374}) can actually be obtained as a direct consequence 
of (\ref{2408}) for L\'evy processes with bounded variation and extended to the general case in a direct way.\\

These results on first passage times will naturally lead us in Section \ref{section2} to the law of the past supremum $\overline{X}_t$ 
of spectrally positive L\'evy processes. In a recent work, Z.~Michna, Z.~Palmowski and M.~Pistorius \cite{mpp} obtained 
the identity 
\begin{equation}\label{7745}
\p(\overline{X}_t>x,X_t\in {\rm d}z)=\int_0^t\frac{x-z}sp_s(z-x)p_{t-s}(x)\,{\rm d}s\,{\rm d}z\,,\;\;x>z\,,
\end{equation}
where $p_t(x)$ is the density of $X_t$. As in \cite{mpp}, our proof of identity 
(\ref{7745}) is based on an application of Kendall's identity. However, we show in Theorem \ref{6823} that a quite simple computation
involving the law of the bridge of the L\'evy process allows us to provide a very short proof of (\ref{7745}). It is first obtained for the 
dual process $-X$ and then derived for $X$ from the time reversal property of L\'evy processes. As a consequence of this result, we 
obtain in Corollary \ref{4582} an integro-differential equation characterizing the entrance law of the excursion measure of the L\'evy 
process $X$ reflected at its infimum. 

\section{Continuous ballot theorem and Kendall's identity}\label{section1}

Let $\mathcal{D}=\mathcal{D}([0,\infty))$ and for $t>0$ let $\mathcal{D}_t=\mathcal{D}([0,t])$ be the spaces of c\`adl\`ag functions 
defined  on $[0,\infty)$ and $[0,t]$, respectively.  Denote by $X$ the canonical process of the coordinates, i.e. for all
$\omega\in\mathcal{D}$ and $s\ge0$ or for all $\omega\in\mathcal{D}_t$ and $s\in[0,t]$, $X_s(\omega)=\omega(s)$. 
The spaces $\mathcal{D}$ and $\mathcal{D}_t$ are endowed with 
their Borel sigma fields $\mathcal{F}$ and $\mathcal{F}_t$, respectively. For any $u\in[0,t]$, we define the family of transformations 
$\theta_u:\mathcal{D}_t\rightarrow\mathcal{D}_t$, as follows:
\begin{equation}\label{def3}
\theta_u (\omega)_s=\left\{\begin{array}{ll}\omega(0)+\omega(s+u)-\omega(u),\;\;\;\;\;\;&\mbox{if}\;\;s< t-u\\
\omega(s-(t-u))+\omega(t)-\omega(u)&\mbox{if}\;\;t-u\leq s\leq t\,.
\end{array}\right.
\end{equation}
The transformation $\theta_u$ consists in inverting the paths $\{\omega(s),\,0\leq s\leq u\}$ and $\{\omega(s),\,u\leq s\leq t\}$ in 
such a way that the new path $\theta_u(\omega)$ has the same values as $\omega$ at times 0 and 1, i.e. 
$\theta_u (\omega)(0)=\omega(0)$ and $\theta_u(\omega)(t)=\omega(t)$. We call $\theta_u$ the {\it shift} at time $u$ over the 
interval [0,t], see the picture below.\\

\begin{center}
\includegraphics[scale=0.7]{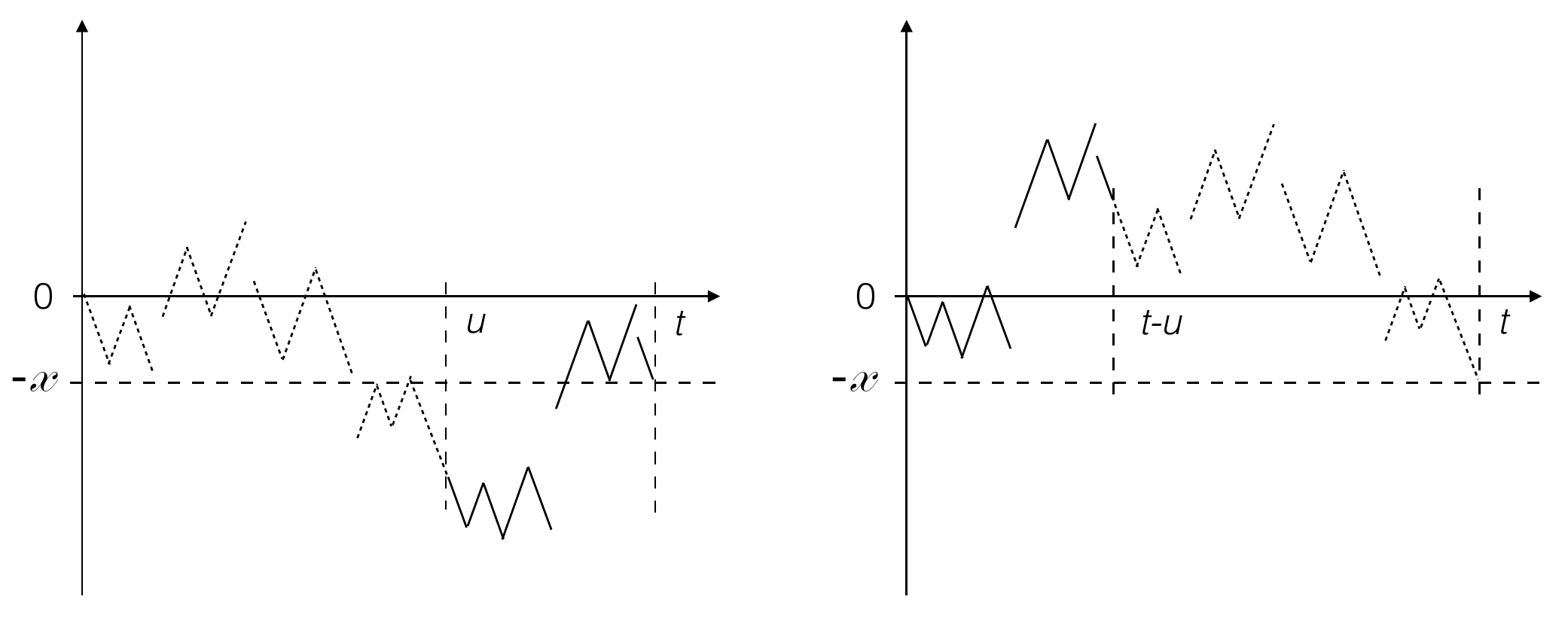} 
\vspace*{.1in}

{A path $\omega$ of $\mathcal{D}_t$ on the left and the shifted path $\theta_u(\omega)$ on the right.}
\end{center}

We say that the process $X=(X_s,\,0\le s\le t)$ has cyclically interchangeable increments under some probability measure $\p$ on 
$(\mathcal{D}_t,\mathcal{F}_t)$ if
\begin{equation}\label{cei}
\theta_u(X)\ed X\,,\quad\mbox{for all $u\in[0,t]$.}\end{equation}
The process $(X,\p)$ will be called a {\it CEI process on $[0,t]$}. Let us note that L\'evy processes are CEI processes. 
We define the past supremum and the past infimum of $X$ before time $s\ge0$ by 
\[\overline{X}_s=\sup_{u\le s}X_u\;\;\;\mbox{and}\;\;\;\underline{X}_s=\inf_{u\le s}X_u\,,\]
this definition being valid for all $s\in[0,t]$ on $\mathcal{D}_t$ and for all $s\ge0$ on $\mathcal{D}$.  
For a stochastic process $Z$ defined on $\mathcal{D}$ or $\mathcal{D}_t$, and $x>0$, we define the first passage time at 
$-x$ by $Z$,
\[T_x(Z)=\inf\{s:Z_s=-x\}\,,\]
with the convention that $\inf\emptyset=\infty$. For the canonical process, we will often simplify this notation by 
setting $T_x:=T_x(X)$.\\

\noindent Here is an extension of Theorem 3 in \cite{bib:ta67}, which is known as the continuous time Ballot theorem.

\begin{theorem}\label{6277} Let $t>0$ and $(X,\p)$ be a CEI process on $[0,t]$ such that $X_0=0$ and $X_t=-x<0$, a.s., 
then
\begin{equation}\label{4865}
\p(T_x=t)=\frac1t\mathbb{E}(\lambda(E_{t,x}))\,,
\end{equation} 
where $\lambda$ is the Lebesgue measure on $\mathbb{R}$ and $E_{t,x}$ is the random set
\[E_{t,x}=\{s\in[0,t]:\mbox{$X_s=\underline{X}_s$ and $X_s\in[\underline{X}_t,\underline{X}_t+x)$}\}\,.\]
In particular if $X$ is of the form $X_s=Y_s-cs$, where $Y$ is a pure jump, non-decreasing CEI process and $c$ is 
some  positive constant, then 
\begin{equation}\label{4866}
\p(T_x=t)=\frac x{ct}\,.
\end{equation} 
\end{theorem}
\begin{proof} First observe that for all $u\in[0,t]$,
\begin{equation}\label{7492}
T_{x}(\theta_u(X))=t\;\,\mbox{if and only if $X_u=\underline{X}_u$ and $X_u\in[\underline{X}_t,\underline{X}_t+x)$}\,.
\end{equation}
This fact is readily seen on the graph of $X$, see for instance the picture above. Then let
$U$ be a uniformly distributed random variable on $[0,t]$ which is independent of $X$ under $\p$. 
The CEI property immediately yields that under $\p$,
\begin{equation}\label{6278}
\theta_U(X)\ed X\,.
\end{equation}
From (\ref{7492}), we obtain $\{T_x(\theta_U(X))=t\}=\{U\in E_{t,x}\}$ and from (\ref{6278}), we derive the equalities,
\begin{eqnarray*}\p(T_x(X)=t)&=&\p(T_x(\theta_U(X))=t)\\
&=&\p(U\in E_{t,x})\\
&=&\frac1t\mathbb{E}(\lambda(E_{t,x}))\,.
\end{eqnarray*}
If $X_s=Y_s-cs$, for a pure jump non-decreasing process $Y$ and a constant $c>0$, then $X$ has bounded variation 
and for all $t\ge0$,
\begin{eqnarray*}
\underline{X}_t&=&\int_0^t\ind_{\{X_u=\underline{X}_u\}}\,{\rm d}X_u\\
&=&\sum_{u\le t}\ind_{\{X_u=\underline{X}_u\}}(Y_u-Y_{u-})-c\int_0^t\ind_{\{X_u=\underline{X}_u\}}\,{\rm d}u.
\end{eqnarray*}
But since $X$ is c\`adl\`ag with no negative jumps, if $u$ is such that $X_u=\underline{X}_u$, then $Y_u=Y_{u-}$.
Therefore $\underline{X}_t=-c\int_0^t\ind_{\{X_u=\underline{X}_u\}}\,{\rm d}u$, so that on the set 
$\{s\in[0,t]:X_s=\underline{X}_s\}$, the Lebesgue measure satisfies $\lambda({\rm d}s)={\rm d}s=-c^{-1}\,{\rm d}\underline{X}_s$, 
and in particular $\lambda(E_{t,x})=x/c$, a.s.\\ 
\end{proof}

\noindent Note that in \cite{bib:ta67}, Theorem \ref{6277} has been proved for {\it separable} processes of the form
$X_s=Y_s-s$, where $Y$ is a pure jump non decreasing CEI process. Separability implies that the past infimum of the 
process is measurable and this property can be considered as the minimal assumption for a CEI process to satisfy the 
ballot theorem. Our proof would still apply, up to slight changes, under this more general assumption. However, since 
this paper is mainly concerned with L\'evy processes, we have chosen the more classical framework of c\`adl\`ag 
processes in which they are usually defined.\\

Let us now focus on three applications of identity (\ref{4865}):\\

\noindent 1. Let $x>0$ and $P$ be some probability measure on $(\mathcal{D}_t,\mathcal{F}_t)$ such that $P(X_t=-x)=1$.
Let $U$ be a random variable on $[0,t]$ which is uniformly distributed and independent of the canonical process 
$X$ under $P$. Denote by $\p$ the law of $\theta_U(X)$ under $P$. 
Then one easily proves that $(X,\p)$ is a CEI process on $[0,t]$. It is also straightforward to show that for all $u\in[0,t]$, 
\begin{equation}\label{3677}
E_{t,x}\circ\theta_u=E_{t,x}+t-u,\;\mbox{mod}\,(t)\,,
\end{equation}
and (\ref{3677}) implies that $\lambda(E_{t,x})=\lambda(E_{t,x}\circ\theta_u)$. It follows from these 
two observations and (\ref{4865}) that
\[P(T_x=t)=\p(T_x=t)\,,\]
which allows us to provide many examples of CEI processes $(X,\p)$ such that $\p(T_x=t)$ is 
explicit. Suppose for instance that under $P$, the canonical process is almost surely equal to the deterministic function 
\[\omega(s)=\left\{\begin{array}{ll}
s^2&\mbox{if $0\le s<\frac t4$}\\
-s^3-x&\mbox{if $\frac t4 \le s< \frac t2$}\\
-(t-s)^3-x&\mbox{if $\frac t2 \le s\le t$}\\
\end{array}\,,
\right.\]
and  set $\underline{\omega}(t)=\inf_{s\le t}\omega(s)$, then 
\[\lambda(\{s\in[0,t]:\mbox{$\omega_s=\underline{\omega}_s$ and 
$\omega_s\in[\underline{\omega}_t,\underline{\omega}_t+x)$}\})=\frac{t}4,\]  
so that from (\ref{4865}),
\[\p(T_x=t)=\frac14.\]

\noindent 2. For our second application, we assume that $(X,\p)$ is the bridge with length $t$ of a L\'evy process from 0 to $-x<0$ 
and we set $\widehat{X}=-X$. Then the process $(\widehat{X},\p)$ is the bridge of the dual L\'evy process from 0  to $x$. By the 
time reversal property of L\'evy processes, 
\[(\widehat{X},\p)=((x+X_{(t-s)-},\,0\le s\le t),\p),\]
where we set $X_{0-}=X_0$. Hence $\p(T_x=t)=\p(\inf_{0\le s\le t}\widehat{X}_s\ge0)$ and from (\ref{4865}),
\[\p(\inf_{0\le s\le t}\widehat{X}_s\ge0)=\p(\sup_{0\le s\le t}X_s\le0)=\frac1t\mathbb{E}(\lambda(E_{t,x})).\]
Integrating this equality over $x$ with respect to the law $\p(X_t\in{\rm d} x)$, this shows that, for the L\'evy process, 
$\sup_{0\le s\le t}X_s\le0$ with positive probability if and only if the set $\{s:\underline{X}_s=X_s\}$ has positive 
Lebesgue measure. Recall that the downward ladder time process is the inverse of the local time defined on this set. Then we 
have recovered the well-know fact that for a L\'evy process, 0 is not regular for $(0,\infty)$ if and only if the downward ladder time 
process has positive drift, see \cite{do}. Note that when $X$ has no negative jumps, this is also equivalent to the fact that it has
bounded variation. \\

\noindent 3. The third application is concerned with $d$-dimensional subordinators, that is $d$-dimensional L\'evy processes whose 
coordinates are non decreasing. The problem originates from a multidimensional extension of the first emptiness problem of reservoirs
raised by D.~Kendall \cite{bib:ken57}. It is considered for instance in Section 3 of \cite{kj} where a rough proof is given and in 
Proposition 6.1 of \cite{js} where Corollary \ref{3472} is stated with a long proof, sometimes difficult to follow. 

Let ${\rm X}_t= (X^{(1)}_t, . . . , X^{(d)}_t)$, $t\ge0$ be a $d$-dimensional subordinator under some probability 
measure $P$, and assume that for each $t\ge0$, the law of ${\rm X}_t$ is absolutely continuous, with density $p_t^{\rm X}({\rm x})$, 
${\rm x}\in[0,\infty)^d$.  We denote by ${\rm x} \cdot {\rm y}$ the usual scalar product of {\rm x} and {\rm y}. Let $\varphi_{\rm X}$ 
be the Laplace exponent of ${\rm X}$, that is
\[E(e^{-{\rm z}\cdot {\rm X}_t}) = e^{-t\varphi_X ({\rm z})},\;\;\;{\rm z}\in[0,\infty)^d.\]
Fix ${\rm r}=(r_1,...,r_d)$ with $r_i\ge0$ and define the one-dimensional process $Z$ by $Z_u=u-{\rm r}\cdot {\rm X}_u$, $u\ge0$. 
Then $(Z,P)$ is clearly a L\'evy process with bounded variation and no positive jumps. Assume that $\sum_{i=1}^dr_iE(X^{(i)}_1)\le1$, 
so that $Z$ does not drift to $-\infty$. Then let us define the first passage time process of $Z$ by $\tau_t=\inf\{u: Z_u=t\}$, $t\ge0$ and 
let ${\rm Y}_t=(Y^{(1)}_t,\dots,Y^{(d)}_t)$ be the $d$-dimensional process whose coordinates are defined by 
\[Y^{(i)}_t=X^{(i)}(\tau_t),\;\;\;t\ge0.\]
It is readily seen that ${\rm Y}$ is a $d$-dimensional subordinator. Its distribution is described as follows.

\begin{corollary}\label{3472}
The law of the subordinator ${\rm Y}$ is absolutely continuous and its density function is given by
\begin{equation}\label{2588}
p^{\rm Y}_t({\rm y})=\frac{t}{t+{\rm r}\cdot{\rm y}}p^{\rm X}_{t+{\rm r}\cdot{\rm y}}({\rm y}),\;\;\;t>0,\,{\rm y}\in[0,\infty)^d. 
\end{equation}
Moreover, the Laplace exponent $\varphi_{\rm Y}({\rm z})$ of the process ${\rm Y}$ satisfies
\begin{equation}\label{2661}
\varphi_{\rm Y}({\rm z})=\varphi_{\rm X}({\rm z}+\varphi_{\rm Y}({\rm z}){\rm r}),\;\;\;{\rm z}\in[0,\infty)^d. 
\end{equation}
\end{corollary}
\begin{proof}
Note that from the definition of $Z$, for all ${\rm y}$ such that $s-{\rm r}\cdot {\rm y}=t$, conditionally on ${\rm X}_s={\rm y}$, 
the process $(Z_u,0\le u\le s)$ has interchangeable increments and satisfies $Z_s=t$, so that from (\ref{4866}), we have 
\begin{equation}\label{6537}
\p(\tau_t=s\,|\,{\rm X}_s={\rm y})=\frac ts.
\end{equation}
Now let us write
\begin{equation}\label{2566}
\p({\rm Y}_t\in {\rm d}{\rm y})=\int_0^\infty \p({\rm X}_s\in {\rm d}{\rm y},\,\tau_t\in {\rm d}s),
\end{equation}
and note that by definition,
\[Z({\tau_t})=t=\tau_t-{\rm r}\cdot{\rm X}({\tau_t}),\]
which shows that the measure $\p({\rm X}_s\in {\rm d}{\rm y},\,\tau_t\in {\rm d}s)$ is carried out by the set 
\[\{s\ge0,{\rm y}\in[0,\infty)^d:s-{\rm r}\cdot {\rm y}=t\}.\]
Then we can write
\begin{eqnarray*}
\p({\rm X}_s\in {\rm d}{\rm y},\,\tau_t\in {\rm d}s)&=&\p({\rm X}_s\in {\rm d}{\rm y},\,\tau_t=s){\delta}_{\{t+{\rm r}\cdot {\rm y}\}}({\rm d}s)\\   
&=&\p({\rm X}_s\in {\rm d}{\rm y})\p(\tau_t=s\,|\,{\rm X}_s={\rm y}){\delta}_{\{t+{\rm r}\cdot {\rm y}\}}({\rm d}s)\\
&=&p^{\rm X}_s({\rm y})\frac ts{\delta}_{\{t+{\rm r}\cdot {\rm y}\}}({\rm d}s){\rm d}{\rm y},
\end{eqnarray*}
where the last equality follows from (\ref{6537}). Then we derive (\ref{2588}) by integrating over $s$ and using (\ref{2566}).

We show identity (\ref{2661}) by noticing that 
\[\tau_t=t+\tilde{\tau}_{{\rm r}\cdot{\rm X}_t}\;\;\;\mbox{and}\;\;\;{\rm X}_{\tau_t}={\rm X}_t+\tilde{\rm X}_{\tilde{\tau}_{{\rm r}\cdot{\rm X}_t}},\]
where $\tilde{\rm X}_s={\rm X}_{t+s}-{\rm X}_t$, $s\ge0$ and $\tilde{\tau}_s=\inf\{u:u-{\rm r}\cdot\tilde{{\rm X}}_u=s\}$. Then we can write 
from the independence between ${\rm X}_t$ and $\tilde{{\rm X}}$, and the fact that $\tilde{{\rm X}}$ has the same law as ${\rm X}$,
\begin{eqnarray*}
\varphi_{\rm Y}({\rm z})&=&E(e^{-{\rm z}\cdot ({\rm X}_t+\tilde{\rm X}_{\tilde{\tau}_{{\rm r}\cdot{\rm X}_t}})})\\
&=&\int_{[0,\infty)^d} e^{-{\rm z}\cdot {\rm x}}E(e^{-{\rm z}\cdot {\rm X}_{\tau_{{\rm r}\cdot{\rm x}}}})p_t^{\rm X}({\rm x})\,{\rm d}{\rm x}\\
&=&\varphi_{\rm X}({\rm z}+\varphi_{\rm Y}({\rm z}){\rm r})\,,
\end{eqnarray*}
which ends the proof.
\end{proof}

From now on we will consider stochastic processes defined on the whole positive half line. In particular, $X$ is now the canonical 
process of $\mathcal{D}$. We shall see in the proof of the following theorem that Kendall's identity is a direct consequence of the 
Ballot theorem. 

\begin{theorem}\label{3555}
Let $(X,\p)$ be a spectrally positive L\'evy process such that $\p(X_0=0)=1$. If $(X,\p)$ is not a subordinator, then the 
following identity between measures: 
\begin{equation}\label{3607}
\p(T_{x}\in {\rm d}t)\,{\rm d}x=\frac xt\p(-X_t\in {\rm d}x)\,{\rm d}t 
\end{equation}
holds  on $(0,\infty)^2$.
\end{theorem}
\begin{proof} Assume first that $X$ has bounded variation, that is $X_t=Y_t-ct$, where $Y$ is a subordinator with no drift and 
$c>0$ is a constant. Let $f$ and $g$ be any two Borel positive functions defined on $\mathbb{R}$. It follows directly from 
(\ref{4866}) by conditioning on $X_t$ that 
$\e(\ind_{\{X_t=\underline{X}_t\}}f(X_t))=-\e\left(\frac{X_t}{ct}f(X_t)\ind_{\{X_t\le0\}}\right)$, so that 
\begin{equation}\label{2273}
\e\left(\int_0^\infty g(t)\ind_{\{X_t=\underline{X}_t\}}f(X_t)\,{\rm d}t\right)=
-\int_0^\infty g(t)\e\left(X_tf(X_t)\ind_{\{X_t\le0\}}\right)\,\frac{{\rm d}t}{ct}.
\end{equation}          
Recall from the end of the proof of Theorem \ref{6277} (applied to processes defined on $[0,\infty)$) that 
${\rm d}t=-c^{-1}\,{\rm d}\underline{X}_t$ on the set $\{t:X_t=\underline{X}_t\}$, so that from the change of variables $t=T_x$,
\begin{eqnarray}\e\left(\int_0^\infty g(t)\ind_{\{X_t=\underline{X}_t\}}f(X_t)\,{\rm d}t\right)&=&
-\e\left(\int_0^\infty g(t)f(X_t)\,c^{-1}\,{\rm d}\underline{X}_t\right)\nonumber\\
&=&\int_0^\infty\e\left( g(T_x)f(-x)\ind_{\{T_x<\infty\}}\right)\,\frac{{\rm d}x}c\,.\label{5151}
\end{eqnarray}
Then (\ref{3607}) follows by comparing the right hand sides of (\ref{2273}) and (\ref{5151}).

Now if $X$ has unbounded variation and Laplace exponent
\[\varphi(\lambda):=\log\e(e^{-\lambda X_1})=-a\lambda+\frac{\sigma^2\lambda^2}2+\int_{(0,\infty)}(e^{-\lambda x}-1+
\lambda x\ind_{\{x<1\}})\,\pi({\rm d}x),\;\;\;\lambda>0,\]
then the spectrally positive L\'evy process $X^{(n)}$ with Laplace exponent 
\[\varphi_n(\lambda):=\log\e(e^{-\lambda X_1^{(n)}})=-a\lambda+\sigma^2(\lambda\sqrt{n}+n(e^{-\lambda/\sqrt{n}}-1))+
\int_{(1/n,\infty)}(e^{-\lambda x}-1+\lambda x\ind_{\{x<1\}})\,\pi({\rm d}x)\] 
has bounded variation and the sequence $X^{(n)}_t$, $n\ge1$ converges weakly toward $X_t$, for all $t$. Recall that $\varphi$ 
and $\varphi_n$ are strictly convex functions. Then let $\rho$ and $\rho_n$ be the largest roots of $\varphi(s)=0$ and $\varphi_n(s)=0$, respectively. Since $X$ and $X^{(n)}$ are not subordinators, $\rho$ and $\rho_n$ are finite and $\rho_n$ tends to $\rho$ as 
$n\rightarrow\infty$. The first passage time $T^{(n)}_x$ by $X^{(n)}$ at $-x$ has Laplace exponent $\varphi_n^{-1}$, 
where $\varphi_n^{-1}$ is the inverse of $\varphi_n$, on $[\rho_n,\infty)$, see chap.~VII in \cite{bib:jb96}. From these 
arguments, $\varphi_n^{-1}$ converges toward the Laplace exponent $\varphi^{-1}$of  $T_x$, so that $T^{(n)}_x$ converges weakly 
toward $T_x$, for all $x>0$. Since $X^{(n)}$ satisfies identity  (\ref{3607}) for each $n\ge1$, so does $X$.\\ 
\end{proof}
 
\noindent Note that if $X=(X_s,\,s\ge0)$ is a stochastic process such that $X=(X_s,\,0\le s\le t)$ is a CEI process for 
all $t>0$, then it has actually interchangeable increments, that is for all $t>0$, $n\ge1$ and for all permutation $\sigma$ of the 
set $\{1,\dots,n\}$,
\[(X_{kt/n}-X_{(k-1)t/n},k=1,\dots,n)\ed (X_{{\sigma(k)t}/{n}}-X_{(\sigma(k)-1)t/n},k=1,\dots,n)\,.\]
A canonical representation for these processes has been given in Theorem 3.1 of \cite{ka}. In particular, conditionally 
on the tail $\sigma$-field $\mathcal{G}=\cap_{t\ge0}\{X_s:s\ge t\}$, the process $X$ is a L\'evy process. By performing again
the proof of Theorem \ref{3555} under the conditional probability $\p(\,\cdot\,|\,\mathcal{G})$ we show that (\ref{3607}) 
is actually valid for all processes with interchangeable increments and no negative jumps. 

\section{The law of the extrema of spectrally one sided L\'evy processes}\label{section2}

Throughout this section we are assuming that,

\begin{itemize} 
\item[$(i)$] the process $(X,\p)$ is a spectrally positive L\'evy process which is not a subordinator and such that $\p(X_0=0)=1$. 
\item[$(ii)$] For all $t>0$, the law $p_t({\rm d}x)$ of $X_t$  is absolutely continuous with respect to the Lebesgue measure. 
We shall denote by $p_t(x)$ its density. 
\end{itemize}
We recall that under assumption $(i)$, 0 is always regular for $(-\infty,0)$ and that 0 is regular for $(0,\infty)$
if and only if $X$ has unbounded variation, see Corollary 5 in Chap.~VII of \cite{bib:jb96}. Let us also mention that condition
$(ii)$ is satisfied for instance if the L\'evy measure $\pi$ of $(X,\p)$ is absolutely continuous and satisfies $\pi(0,\infty)=\infty$, 
see  Theorem 27.7 in \cite{sa}.\\

Now we briefly recall the definition of bridges of L\'evy processes. The law $\p^{t}_{y}$ of the 
bridge from $0$ to $y\in\mathbb{R}$, with length $t>0$ of the L\'evy process $(X,\p)$ is 
a regular version of the conditional law of $(X_s,\,0\le s\le t)$ given $X_t=y$, under $\p$.  It satisfies  $\p^{t}_{y}(X_0=0,X_t=y)=1$ 
and for all $s<t$, this law is absolutely continuous with respect to $\p$ on ${\mathcal F}_s$, with density $p_{t-s}(y-X_s)/p_t(y)$, i.e.
\begin{equation}\label{4573}
\p_{y}^t(\Lambda)=\e\left(\ind_{\Lambda}\frac{p_{t-s}(y-X_s)}{p_t(y)}\right)\,,\;\;\;\mbox{for all $\Lambda\in{\mathcal F}_s$}\,.
\end{equation}
Note that from Theorem (3.3) in \cite{sh}, $p_t(y)>0$, for all $t>0$ and $y\in\mathbb{R}$ if and only if for all $c\ge0$, the process
$(|X_t-ct|,\,t\ge0)$ is not a subordinator. But from assumptions $(i)$ and $(ii)$, the later condition is always satisfied in our framework.\\ 

Formula (\ref{1574}) below was proved in Theorem 2.4 in \cite{mpp}, see also \cite{mi} and Theorem 12 in \cite{ko} for the stable case. 
Here we first prove an analogous formula for the dual process in (\ref{1534}) from which (\ref{1574}) is immediately derived.

\begin{theorem} \label{6823} The laws of $(\overline{X}_t,X_t)$ and $(\underline{X}_t,X_t)$ admit the following expressions,  
\begin{eqnarray}
\p(\underline{X}_t<-x,X_t\in {\rm d}z)&=&\int_0^t\frac{x}sp_s(-x)p_{t-s}(x+z)\,{\rm d}s\,{\rm d}z,\;\;-x\le z,\;\;x>0,\label{1534}\\
\p(\overline{X}_t>x,X_t\in {\rm d}z)&=&\int_0^t\frac{x-z}sp_s(z-x)p_{t-s}(x)\,{\rm d}s\,{\rm d}z,\;\;x>z,\;\;x\ge0.\label{1574}
\end{eqnarray}
The process $(X,\p)$ has bounded variation if and only if for all $t\ge0$, $\p(\overline{X}_t=0)>0$ and $\p(\underline{X}_t=X_t)>0$.
In this case, the expressions $(\ref{1534})$ and $(\ref{1574})$ can be completed by the following one,
\begin{equation}\label{9059}
\p(\overline{X}_t=0,\,X_t\in {\rm d}z)=\p(\underline{X}_t=X_t\in {\rm d}z)=-\frac z{ct}p_t(z)\,{\rm d}z,\;\;\;z<0\,,
\end{equation}
where $-c$ is the drift of $X$.
\end{theorem}
\begin{proof} From (\ref{4573}) applied at the stopping time $T_{x}=\inf\{s:X_s=-x\}$, we obtain 
\begin{eqnarray*}
\p_{z}^t(\underline{X}_t<-x)&=&\p_{z}^t(T_{x}<t)\\
&=&\e\left(\ind_{\{T_{-x}<t\}}\frac{p_{t-T_{x}}(z-X_{T_{x}})}{p_t(z)}\right)\\
&=&\e\left(\ind_{\{T_{-x}<t\}}\frac{p_{t-T_{x}}(x+z)}{p_t(z)}\right)\,,
\end{eqnarray*}
where in the third equality we used the fact that $X$ has no negative jumps. Recalling the definition of the law $\p_{z}^t$, we 
derive from the above equality that
\begin{eqnarray*}
\p(\underline{X}_t<-x,\,X_t\in {\rm d}z)&=&\e(\ind_{\{T_{x}<t\}}p_{t-T_{x}}(x+z))\,{\rm d}z\\
&=&\int_0^t\p(T_{-x}\in {\rm d}s)p_{t-s}(x+z)\,{\rm d}z\,.
\end{eqnarray*}
Then (\ref{1534}) is obtained by plunging Kendall's identity (\ref{3607}) in the right hand side of the above equality.

Identity (\ref{1574}) follows by replacing $x$ by $x-z$ in (\ref{1534}) and by applying the time reversal property of 
L\'evy processes, that is under $\p$, 
\begin{equation}\label{3683}
(X_s,0\le s< t)\ed(X_t-X_{(t-s)-},\,0\le s< t)\,.
\end{equation}

If $(X,\p)$ has bounded variation, then 0 is not regular for the half line $(0,\infty)$, so that for all $t\ge0$, 
$\p(\overline{X}_t=0)>0$ and $\p(\underline{X}_t=X_t)>0$, where the second inequality follows from the time reversal property 
(\ref{3683}). Then (\ref{9059}) follows directly from (\ref{2273}).
\end{proof}

We can derive from Theorem \ref{6823} a series of immediate corollaries. First we obtain the distribution functions of 
$\overline{X}_t$ and $\underline{X}_t$ by integrating identity (\ref{1534}), (\ref{1574}) and (\ref{9059}) over $z$.

\begin{corollary}  For all $t\ge0$ and $x>0$, 
\begin{eqnarray}\label{3442}
\p(\underline{X}_t<-x)&=&\int_0^t\p(X_{t-s}>0)p_{s}(-x)\,\frac{{\rm d}s}s+\p(X_t<-x)\,,\label{1725}\\
\p(\overline{X}_t>x)&=&\int_0^t\e(X_s^-)p_{t-s}(x)\,\frac{{\rm d}s}s+\p(X_t>x)\,.\label{3442}
\end{eqnarray}
If $X$ has bounded variation with drift $-c$, then for all $t>0$, 
\begin{equation}\label{1745}
\p(\overline{X}_t=0)=-\frac{\e(X_t\ind_{\{X_t\le0\}})}{ct}\,.
\end{equation}
\end{corollary}
\noindent Note that we can derive from (\ref{3607}) the following simpler expression for the distribution function of  
$\underline{X}_t$.
\begin{equation}\label{5632}
\p(\underline{X}_t<-x)=\int_0^t xp_s(-x)\,\frac{{\rm d}s}s\,.
\end{equation}
There exists a huge literature on the law of the extrema of spectrally one sided L\'evy processes. First explicit results
were obtained for processes with bounded variation in \cite{bib:tak67}. Then the stable case 
has received particular attention. In \cite{bi} it is proved that $\underline{X}_t$ has a Mittag-Leffler distribution. Then
the law of $\overline{X}_t$ was first characterized in \cite{bdp} and was followed by more explicit forms in \cite{hk},
\cite{mi} and Theorem 12 in \cite{ko}. In the general case, one is tempted to derive expressions for the density of the 
extrema by differentiating (\ref{1725}), (\ref{3442}) and (\ref{5632}) but proving conditions allowing us to do so is an
open problem. Only some estimates of these densities have been given in \cite{bib:chm14} and \cite{bib:kmr12}.\\

Multiplying each side of identities (\ref{1725}) and (\ref{3442}) by $e^{-\lambda x}$ or $x^n$ and  integrating we
obtain the following other immediate consequence of Theorem \ref{6823}.

\begin{corollary} The Laplace transform of $\overline{X}_t$ and $\underline{X}_t$ are given for 
$\lambda\ge0$ by 
\begin{eqnarray*}
\e(e^{\lambda\underline{X}_t})=-\lambda\int_0^t\p(X_{t-s}>0)\e(e^{\lambda X_{s}}\ind_{\{X_{s}\le0\}})\frac{{\rm d}s}s+
\e(e^{\lambda X_t}\ind_{\{X_{t}\le0\}})+\p(X_t>0)\,,\label{3732}\\
\e(e^{-\lambda\overline{X}_t})=-\lambda\int_0^t\e(X_s^-)\e(e^{-\lambda X_{t-s}}\ind_{\{X_{t-s}>0\}})\frac{{\rm d}s}s+
\e(e^{-\lambda X_t}\ind_{\{X_{t}>0\}})+\p(X_t\le0)\,.\label{3722}
\end{eqnarray*}
Assume moreover that $X$ admit a moment of order $n\ge1$. Then $\overline{X}_t$ and $\underline{X}_t$ admits a moment of 
order $n$ and the later are given by,
\begin{eqnarray*}
\e((-\underline{X}_t)^n)=n\int_0^t\p(X_{t-s}>0)\e((-X_{s})^{n-1}\ind_{\{X_{s}<0\}})\frac{{\rm d}s}s+\e((X_t^-)^n)\,,\label{2383}\\
\e(\overline{X}_t^n)=n\int_0^t\e(X_s^-)\e(X_{t-s}^{n-1}\ind_{\{X_{t-s}\ge0\}})\frac{{\rm d}s}s+\e((X_t^+)^n)\,.\label{2583}
\end{eqnarray*}
\end{corollary}

\noindent Then for $\lambda\ge0$ and $z<0$, define the Laplace transform
of the function $t\mapsto t^{-1}p_t(z)$ by  
\[\varphi(\lambda,z)=\int_0^\infty e^{-\lambda t}t^{-1}p_t(z)\,{\rm d}t\,.\]

\begin{corollary} The Laplace transform $\varphi(\lambda,z)$  satisfies the equation
\begin{equation}\label{4732}
\varphi(\lambda,z)=\varphi(0,z)+e^{-z \Phi(\lambda)}\,,\;\;\lambda\ge0\,,\;z<0\,,
\end{equation}
where $\Phi(\lambda)=\int_0^\infty (1-e^{-\lambda t})t^{-1}p_t(0)\,{\rm d}t$.
\end{corollary}
\begin{proof}
Letting $x=0$ in identity (\ref{1574}), we obtain for $z<0$, $p_t(z)=\int_0^t\frac{-z}sp_s(z)p_{t-s}(0)\,{\rm d}s$. 
Taking the Laplace transform
of each side of this identity gives
\[\frac{\partial}{\partial \lambda}\varphi(\lambda,z)=-z\varphi(\lambda,z)\int_0^\infty e^{-\lambda t}p_t(0)\,{\rm d}t\,,\]
whose solution is given by (\ref{4732}).\\
\end{proof}

\noindent Recall from \cite{bib:lc11} and \cite{bib:chm14} the definition of the entrance laws $q_t({\rm d}x)$ 
(resp.~$q_t^*({\rm d}x)$) of the excursions reflected at the supremum (resp.~at the infimum) of $X$. 
Both reflected processes $\overline{X}-X$ and $X-\underline{X}$ are homogeneous Markov processes. 
We denote by $n$ and $n^*$ the characteristic measures of the corresponding Poisson point processes of excursions 
away from 0, see \cite{bib:lc11}. Then $q_t({\rm d}x)$ and $q_t^*({\rm d}x)$ are defined by 
\[n(f(X_t),\,t<\zeta)=\int_{[0,\infty)} f(x)q_t({\rm d}x)\;\;\;\mbox{and}\;\;\;n^*(f(X_t),\,t<\zeta)=\int_{[0,\infty)} f(x)q_t^*({\rm d}x)\,,\]
where $\zeta$ denotes the life time of the excursions and $f$ is any positive Borel function. 
We also recall that if $p_t({\rm d}x)$ is absolutely continuous, then so are $q_t({\rm d}x)$ and $q_t^*({\rm d}x)$, 
see part (3) of Lemma 1, p.~1208 in \cite{bib:lc11}. We will denote the corresponding densities by $q_t(x)$ and 
$q_t^*(x)$. Thanks to the absence of negative jumps, the entrance law $q_t({\rm d}x)$ can be related to the law of $X_t$ 
through the relation,
\begin{equation}\label{3246}
q_{t}(x)=\frac{x}{t}p_{t}(-x)\,, 
\end{equation}
which is valid for all $t>0$ and $x\ge0$, see (5.10), p.1208 in \cite{bib:lc11}. We now use this fact and Theorem \ref{4573}, 
in order to describe the entrance law $q_t^*({\rm d}x)$. 

\begin{corollary}\label{4582} The entrance law $q_t^*(x)$ satisfies the equation,
\begin{equation}\label{3562}
\int_0^t\frac{x-z}{t-s}p_{t-s}(z-x)q^*_s(x)\,{\rm d}s=-\frac{{\rm d}}{{\rm d}x}\int_0^t\frac{x-z}{t-s} p_{t-s}(z-x)p_{s}(x)\,{\rm d}s\,,
\end{equation}
 for all $t>0$, $x>0$ and $z<x$.
\end{corollary}
\begin{proof} Let us recall that from Theorem 6 in \cite{bib:lc11}, the law of the couple $(\overline{X}_t,X_t)$ is given in terms 
of $q_t$ and $q_t^*$ as follows,
\begin{equation}\label{3569}
\p(\overline{X}_t\in {\rm d}x,\,X_t\in{\rm d}z)=\int_0^tq_s^*(x)q_{t-s}(x-z)\,{\rm d}s\,{\rm d}x\,{\rm d}z\,,
\end{equation}
for $x>0$ and $z<x$. Then plunging (\ref{3246}) into (\ref{3569}) and comparing this expression with (\ref{1574}) where we 
performed the time change $s\rightarrow t-s$ and we differentiated in $x>0$, we obtain (\ref{3562}).\\
\end{proof}

Let us finally point out that actually Theorem 6 in \cite{bib:lc11} gives the following disintegrated version of (\ref{3569}),
\begin{equation}\label{2475}
\p(g_t\in {\rm d}s,\overline{X}_t\in {\rm d}x,\,X_t\in{\rm d}z)=q_s^*(x)q_{t-s}(x-z)\ind_{[0,t]}(s)\,{\rm d}x\,{\rm d}z\,,
\end{equation}
on $(0,\infty)^2\times\mathbb{R}$, where $g_t$ is the unique time at which the past supremum of $(X,\p)$ occurs on $[0,t]$. 
This result suggests a possibility of disintegrating also (\ref{1574}) according to the law of $g_t$. Then comparing this disintegrated 
form with (\ref{2475}) would provide a means to obtain an expression for the density $q^*_t(x)$ in terms of $p_t(x)$.  
However, this problem remains open.\\

\vspace*{0.2in}

\noindent {\bf Acknowledgement}: We are grateful to Alexey Kuznetsov who pointed out to us the statement of 
Corollary \ref{3472} and related references.

\end{document}